\documentclass[11pt,a4paper]{amsart}
\usepackage[left=3.5cm,right=3.5cm,top=3.5cm,bottom=3.5cm]{geometry}
\usepackage{enumerate}
\usepackage{leftidx}
\usepackage{mathtools,setspace}
\usepackage{amsmath,amscd,amssymb,amsthm,mathrsfs,stmaryrd}
\usepackage[arrow,matrix]{xy}
\usepackage{graphicx,color,tikz}
\usepackage[colorlinks,linkcolor=blue,anchorcolor=blue,citecolor=blue,backref=page]{hyperref}
\usepackage{mathabx}
\usetikzlibrary {graphs}

\usetikzlibrary{positioning,bending}
\usetikzlibrary{arrows.meta}

\DeclareMathOperator{\gal}{Gal}

\DeclareMathOperator{\ab}{ab}
\DeclareMathOperator{\GL}{GL}

\DeclareMathOperator{\en}{End}

\theoremstyle{definition}
\newtheorem{definition}{Definition}[section]
\newtheorem{example}[definition]{Example}
\newtheorem{remark}[definition]{Remark}

\theoremstyle{plain}
\newtheorem{theorem}[definition]{Theorem}
\newtheorem{corollary}[definition]{Corollary}
\newtheorem{lemma}[definition]{Lemma}
\newtheorem{proposition}[definition]{Proposition}
\newtheorem{conjecture}[definition]{Conjecture}
\newtheorem*{theorem*}{Theorem}

\newcommand{\F}{{ \mathbb F }}
\newcommand{\N}{{ \mathbb N }}
\newcommand{\Q}{{ \mathbb Q }}
\newcommand{\Z}{{ \mathbb Z }}

\renewcommand{\P}{{ \mathbb P }}

\newcommand{\p}{{ \mathfrak p }}

\renewcommand{\O}{{ \mathcal O }}
\newcommand{\Fab}{F^{\text{ab}}}
\def\ab{\text{ab}}

\author[A. Ferraguti]{Andrea Ferraguti}
\thanks{\textit{Andrea Ferraguti:} \href{mailto:and.ferraguti@gmail.com}{\texttt{and.ferraguti@gmail.com}}}

\author[P. Ingram]{Patrick Ingram}
\thanks{\textit{Patrick Ingram:} \href{mailto:pingram@yorku.ca}{\texttt{pingram@yorku.ca}}}

\author[C. Pagano]{Carlo Pagano}
\thanks{\textit{Carlo Pagano:} \href{mailto:carlopagano@google.com}{\texttt{carlopagano@google.com}}}

\makeatletter
\renewcommand{\@setauthors}{%
  \begingroup
  \trivlist
  \centering\footnotesize
  \@topsep30\p@\relax
  \advance\@topsep by -\baselineskip
  \item\relax
  \begin{tabular*}{\textwidth}[t]{@{\extracolsep{\fill}}ccc@{}}
    \MakeUppercase{Andrea Ferraguti} & \MakeUppercase{Patrick Ingram} & \MakeUppercase{Carlo Pagano} \\[3pt]
    \textit{Universit\`a di Torino} & \textit{York University} & \textit{Concordia University} \\
    & & \textit{Google DeepMind}
  \end{tabular*}
  \endtrivlist
  \endgroup
}
\makeatother

\title[]{Abelian dynamical Galois groups over global function fields}

\begin{document}
\maketitle
\begin{abstract}
We establish a function field analogue of a recent conjecture of Andrews--Petsche. Our main result characterizes abelian dynamical Galois groups to be precisely the isotrivial ones, whenever the degree of the polynomial is smaller than $p$, the characteristic of the field.

The proof of this characterization is achieved in four independent steps as follows:
$$\text{Abelian} \implies \text{Finite ramification} \implies \text{PCF map} \implies \text{Isotrivial map} \implies \text{Isotrivial pair}, 
$$
and works more generally for maps with a superattracting fixed point. 

We observe that this chain of implications is sharp in the degree: as soon as one reaches $p$ there are new non-isotrivial examples coming from Drinfeld modules. We propose a full conjectural classification in all degrees, taking into account all of the new exotic examples coming from Drinfeld modules and their associated Latt\`es maps.

\end{abstract}
\section{Introduction}
This paper is devoted to establish a function field analogue of the following conjecture proposed in $2020$ by Andrews--Petsche \cite{andrews}. 
\begin{conjecture}\label{ap_conjecture}
    Let $K$ be a number field, $f\in K[x]$ have degree $d\ge 2$ and $\alpha\in K$ be non-exceptional for $f$. Then $G_\infty(f,\alpha)$ is abelian if and only if $(f,\alpha)$ is $K^{\text{ab}}$-conjugate to $(x^d+\zeta)$ or $(\pm T_d(x),\zeta+\zeta^{-1})$, where $T_d$ is the Chebishev polynomial of degree $d$ and $\zeta$ is a root of unity.
\end{conjecture}

Our main result in this direction is the following.

\begin{theorem}\label{main_thm}
Let $d\ge 2$ be an integer, and let $p$ be a prime with $p>d$. Let $F$ be a global function field of characteristic $p$, let $f\in F(x)$ be a rational function having a superattracting fixed point, and let $\alpha\in F$. Suppose the pair $(f,\alpha)$ is not exceptional. Then $G_\infty(f,\alpha)$ is abelian if and only if $(f,\alpha)$ is $F^{\text{ab}}$-conjugate to a constant pair.
\end{theorem} 

We now briefly summarize the logic of the proof of Theorem \ref{main_thm}. This consists of four independent steps that can be diagrammatically represented as
$$\text{Abelian} \implies \text{Finite ramification} \implies \text{PCF map} \implies$$
$$\implies \text{Isotrivial map} \implies \text{Isotrivial pair}. 
$$
In \emph{each} of these steps the assumption $d<p$ plays a crucial role in a somewhat different fashion. 

In the first step we adapt an argument of Ferraguti--Ostafe--Zannier \cite{ferostzan} to global function fields. The key idea here is to show an upper bound of the size of the largest prime ramifying in terms of the height of a discriminant polynomial purely depending on $f$. Here the assumption $d<p$ is used in verifying that such discriminants cannot be constantly $0$. 

In the second step we adapt an argument of Bridy--Ingram--Jones--Juul--Levy--Manes--Rubinstein-Salzedo--Silverman \cite{bridy} to global function fields. As this argument relies crucially on Faltings' theorem, some care needs to be paid in its function field analogue to isotrivial curves. We use $d<p$ here as well, in checking that a local argument with Newton polygons provided in \cite{bridy} goes through. 

The third step is provided by a version of Thurston rigidity theorem proved by the second named author \cite{ingram2}, and the assumption $d<p$ here is vital, as otherwise one has the families $x^p+tx$. 

For the fourth step, the situation is as follows. One has a pair $(f,\alpha)$, where $f$ is a function with constant coefficients, $G_{\infty}(f,\alpha)$ is abelian, and one wants to conclude that $\alpha$ is also constant. Thanks to the product formula, it suffices to show that $\alpha$ cannot have negative valuation at any place.  The key idea here is to use local \emph{Böttcher coordinates} at every such place: these coordinates are able to detect a large Kummer subextension in the local arboreal field, as soon as the base point $\alpha$ lands sufficiently close to infinity. When $f$ is already isotrivial, happily, ``sufficiently close" simply means $|\alpha|>1$. Hence to grant all of the local Galois groups a chance to be abelian already forces $|\alpha| \leq 1$ at every place, which by the product formula forces $\alpha$ to be constant. Here the assumption $d<p$ is crucially used to show the existence of Böttcher coordinates.

The idea of exploiting the Galois invariance of Böttcher coordinates to study local arboreal representations, has been introduced by the second named author for polynomials in \cite{ingram1}, where the entire tree is recognized to be a Kummer tree after the Böttcher change of variable. We adapt that work to superattracting maps, where instead it is only a canonical subtree coming from intersecting the tree with the Böttcher's region, and whose degree equals the local degree of the map at $\infty$. This gives a Kummer subrepresentation in the arboreal field. 

The idea of identifying a systematic \emph{root extraction procedure} played already a role in the context of Conjecture \ref{ap_conjecture}. This can be found in the joint work of the first and third author settling the conjecture for unicritical PCF polynomials with periodical critical orbit \cite{ferpag2}. For a deeper study of this this subject we refer to \cite{Hamblen-Jones}. While in these previous works the basic machinery of root extraction is \emph{global}, and based on special algebraic identities available for functions with periodic orbits, in the present paper the machinery of root extraction is genuinely \emph{local}, and instead provided by the analytic change of variable given by Böttcher coordinates. 

Finally we explore what happens from degree $p$ onwards. New phenomena emerge here, through the presence of maps associated with endomorphisms of CM Drinfeld modules and corresponding Latt\`es maps. We make the following conjecture; we denote by $C_{f,\alpha}$ the set of critical points of $f$ whose forward orbit contains $\alpha$.
\begin{conjecture} \label{our conjecture: intro}
	Let $F/\F_q(T)$ be a finite extension, let $f\in F[x]$ with $f'\not\equiv 0$ have degree $\ge 2$ and be such that $C_{f,\alpha}\subseteq F^{\text{sep}}$, and let $\alpha\in F$. Then $G_\infty(f,\alpha)$ is virtually abelian if and only if one of the following holds:
	\begin{enumerate}
		\item The pair $(f,\alpha)$ is $\overline{F}$-conjugate to a constant pair.
		\item The pair $(f,\alpha)$ is $\overline{F}$-conjugate to a pair $(\psi,\beta)$ where $\psi$ is a Drinfeld-Latt\`es map for a CM Drinfeld module $\phi$ of generic characteristic and $\beta$ is preperiodic for $\psi$.
	\end{enumerate} 
\end{conjecture}

For more background on the terminology see Section \ref{Section: Drinfeld}. Conjecture \ref{our conjecture: intro} can be thought as the full function field analogue of Conjecture \ref{ap_conjecture}. Theorem \ref{main_thm}, when specialized to polynomials, shows that Conjecture \ref{our conjecture: intro} holds in degree smaller than the characteristic. In Theorem \ref{abelian_quotients} we prove that Conjecture \ref{our conjecture: intro} holds also when $\psi$ is either a Drinfeld map or a Drinfeld-Latt\`es map. 

Finally, we remark that the reduction step provided by Thurston to show 
$$\text{PCF map} \implies \text{Isotrivial map}
$$
can be used to prove that Conjecture \ref{ap_conjecture} implies the same conjecture over function fields of curves defined over number fields, as in the following theorem. We do not include a proof since it would deviate too much from the paper's setup. 
\begin{theorem*}
	Assume Conjecture \ref{ap_conjecture}. Let $K$ be a number field, let $C$ be a smooth projective curve defined over $K$ and let $F$ be its function field. Let $f\in F[X]$ be a polynomial of degree $d\ge 2$ and let $\alpha\in F$. Then $G_\infty(f,\alpha)$ is abelian if and only if $f$ is $F^{\text{ab}}$-conjugate to $(x^d,\zeta)$ or $(\pm T_d(x),\zeta+\zeta^{-1})$.
\end{theorem*}
 For a survey of what has been done until recently on Conjecture \ref{ap_conjecture}, we refer to \cite{fer}. We mention here only two further works that came out afterwards. In \cite{ferpag2} the first and third author proved Conjecture \ref{ap_conjecture} for all unicritical polynomials over at most quadratic number fields, and for unicritical polynomials with periodic critical orbit over any number field. Recently, Leung and Petsche proved in \cite{petsche} Conjecture \ref{ap_conjecture} for all pairs $(f,\alpha)$ where $f$ is PCF and $\alpha$ is not $f$-preperiodic, over any number field, using equidistribution ideas.

While this work was prepared, the first and third author have developed a proof for Conjecture \ref{ap_conjecture}, established independently also by Zhuchao Ji, Jiarui Song and Junyi Xie. Both approaches rely on equidistribution results combined with rigidity results for semi-conjugacy classes of rational maps. We expect these techniques to be eventually transferrable in positive characteristics, but some additional work will be needed to capture the greater variety of phenomena generated by Drinfeld modules in our case. 
\subsection{Layout of the paper}
In Section \ref{Sct: ab impl fin ram} we perform the step
$$\text{Abelian} \implies \text{Finite ramification}.
$$
In Section \ref{Sct: fin ram implies PCF} we do the step
$$\text{Finite ramification} \implies \text{PCF map}.
$$
In Section \ref{subsection: Bottcher as formal} we introduce Böttcher's coordinates for maps with a superattracting fixed point, and in Section \ref{subsection: root extraction} we give the key tool to do the step
$$\text{Isotrivial map} \implies \text{Isotrivial pair}.
$$
This section consists of an extension of some of the results of \cite{ingram1}, to power series $f$ with a superattracting fixed point of order $e \geq 2$. We also opted for a different exposition where the existence and uniqueness of Böttcher's coordinates along with their classical limiting shape
$$\lim_{n \to \infty} (f^n(x))^{\frac{1}{e^n}},
$$
are all derived in one stroke from Banach's fixed point theorem. 

In Section \ref{Sct: Pcf impl isotriv} we give a precise version of
$$\text{PCF map} \implies \text{Isotrivial map},
$$
and references for the proof, after \cite{ingram2}.

In Section \ref{Section: Proof} we collect all the work of the previous sections and we articulate the proof of Theorem \ref{main_thm}. 

In Section \ref{Section: Drinfeld} we investigate the case of maps coming from Drinfeld modules, we propose Conjecture \ref{our conjecture} and we prove Theorem \ref{abelian_quotients}
\subsection{Human--AI collaboration} This paper has been ideated in Summer $2023$ and executed for the largest part in $2024$. At the ideation level, we have used artificial intelligence only for some steps of the proof of Theorem \ref{abelian_quotients} where public models of Gemini have provided useful references. We have drafted this work without use of AI. We have proof-read the work using \emph{Astra} and \emph{Fable}.

\subsection*{Acknowledgments} The first author has been partially supported by the ``National Group for Algebraic and Geometric Structures, and their Applications" (GNSAGA - INdAM).

 \section{Abelian implies finite ramification} \label{Sct: ab impl fin ram}
 We let $p$ be a prime, $F$ be a global function field of characteristic $p$, $f\in F(x)$ be of degree $d$ with $p\nmid d$ and $\alpha \in F$. Define
 $$F_\infty^{\text{ab}}(f,\alpha)\coloneqq F(f^{-\infty}(\alpha)\cap F^{\text{ab}}).$$
 \begin{theorem}\label{abelian_implies_fr}
     The extension $F_\infty^{\text{ab}}(f,\alpha)/F$ ramifies only at a finite number of primes.
 \end{theorem}
 \begin{proof}
 	If $\alpha$ is exceptional, the statement is obvious. Hence from now on we assume that $\alpha$ is not exceptional.
 	
     Up to replacing $F$ by a finite extension and conjugating the pair $(f,\alpha)$ over $F$ we can assume that $f=g/h$ with $g,h\in F[x]$ coprime, $g$ monic and $\deg g>\deg h$. Let $S$ be a finite set of primes of $F$ such that $\alpha$ and all the coefficients of $g$ and $h$ are $S$-integral, and such that $F/\mathbb{F}_p(T)$ is unramified outside of $S$. By induction it is easy to see that every element in $f^{-\infty}(\alpha)$ is integral at all primes not above elements of $S$.

     Now consider $\Delta(t)\coloneqq \text{Disc}(g(x)-th(x))\in F[t]$, where $t$ is transcendental over $F$. Since $p\nmid d$, $\Delta(t)$ cannot be identically zero. Therefore there are only finitely many $\gamma\in f^{-\infty}(\alpha)$ that are such that $\Delta(\gamma)=0$. For every such separable $\gamma$, let $S_\gamma$ be the set of primes of $F(\gamma)$ that ramify in $F(f^{-1}(\gamma))$. Up to adding a finite number of primes to $S$, we can assume that $S$ contains all primes lying below a prime belonging to some of the $S_\gamma$.
     
     Let $\p$ be a prime of $F$ that lies outside of $S$ and is ramified in $F_\infty^{\text{ab}}(f,\alpha)$. Clearly there exists $\beta \in f^{-\infty}(\alpha)\cap \Fab$ such that $\p$ ramifies in $F(\beta)$ but is unramified in $F(f(\beta))$. By our assumptions on $S$, we have that $\Delta(f(\beta))\ne 0$. Since the extension $F(\beta)/F$ is abelian, so is $F(f(\beta))/F$, and therefore every prime $\mathfrak P$ of $F(f(\beta))$ that divides $\p$ ramifies in $F(\beta)$. It follows that $\Delta(f(\beta))$, that is integral at all primes of $F(f(\beta))$ that lie above primes outside of $S$, has positive valuation at every such $\mathfrak P$. We normalize the valuation at $\mathfrak P$ so that $|q(t)|_{\mathfrak P}=1/p^{\deg q(t)}$, where $q(t)\in \F_p[t]$ is the irreducible polynomial corresponding to the place of $\F_p(t)$ lying below $\p$ (that we denote by $\mathfrak q$), and we let $\Delta\coloneqq \Delta(f(\beta))$ to ease the notation. Then
     $$|\Delta|_{\mathfrak P}\le \frac{1}{p^{\deg q(t)}} \mbox{ for every }\mathfrak P\mid \p,$$
     and hence
     $$\prod_{\mathfrak P\mid \p}|\Delta|_{\mathfrak P}^{[F(f(\beta))_{\mathfrak P}:F_\p]}\le \frac{1}{p^{\deg q(t)\cdot[F(f(\beta)):F]}}.$$
     Since at every place dividing $\p$ we have that $\Delta$ has positive valuation, we can isolate the contribution of these primes in computing the Weil height of $\Delta$, that is:
     $$h(\Delta)=h(\Delta^{-1})\ge \frac{1}{[F(f(\beta)):\F_p(t)]}\sum_{\mathfrak P\mid \p}[F(f(\beta))_{\mathfrak P}:\F_p(t)_{\mathfrak q}]\log^+(|\Delta|_{\mathfrak P}^{-1})=$$
     $$=\frac{[F_\p:\F_p(t)_{\mathfrak q}\rrbracket}{[F(f(\beta)):\F_p(t)]}\sum_{\mathfrak P\mid \p}[F(f(\beta))_{\mathfrak P}:F_\p]\log^+(|\Delta|_{\mathfrak P}|^{-1})\ge$$
     $$\ge \frac{1}{[F(f(\beta)):\F_p(t)]}\log\left(\prod_{\mathfrak P\mid p}|\Delta|_{\mathfrak P}^{-[F(f(\beta))_{\mathfrak P}:F_\p]}\right)\ge \frac{\deg q(t)\log p}{[F:\F_p(t)]}.$$
     On the other hand, $\Delta(t)$ is a polynomial in $t$ whose coefficients only depend on $f$, and therefore
     $$h(\Delta)\le C_1h(f(\beta))+C_2$$
     for some constants $C_1,C_2$ that depend only on $f$. Moreover, since $f(\beta)\in f^{-\infty}(\alpha)$ then $h(f(\beta))<C_3$ for some constant $C_3$ that depends on $f$ and on $\alpha$. We can therefore conclude that
     $$\deg q(t)\le \frac{C_4[F:\F_p(t)]}{\log p}$$
     for some constant $C_4$ that depends on $f$ and $\alpha$. The claim follows immediately.
 \end{proof}

 \section{Finite ramification implies PCF} \label{Sct: fin ram implies PCF}

 Let $p$ be a prime and $F$ be a global function field of characteristic $p$. Let $f\in F(x)$ be of degree $d$ such that $p>d$. For every $n\ge 1$, we let $f^n(x)=\frac{F_n(x)}{G_n(x)}$ with $F_n(x),G_n(x)\in F[x]$ coprime polynomials. In this section, we prove the following theorem, that is the analogous of \cite[Theorem 5]{bridy} in positive characteristic.
 \begin{theorem}\label{finite_ramification_implies_pcf}
    Suppose that $F_\infty(f,\alpha)/F$ ramifies at finitely many places and that $\alpha$ is not exceptional for $f$. Then $f$ is PCF.
\end{theorem}
 
 We start with a lemma that is an immediate consequence of the results of \cite{hindes1}. Choose coprime homogeneous forms $P,Q\in F[X,Y]$ of degree $d$
 such that $f=[P:Q]$. Put $P_0=X$, $Q_0=Y$, and define recursively
 $$P_{n+1}=P(P_n,Q_n), \qquad Q_{n+1}=Q(P_n,Q_n),$$
 so that $f^n=[P_n:Q_n]$. We use the compatible affine presentation
 $$F_n(x)=P_n(x,1), \qquad G_n(x)=Q_n(x,1).$$
 If $\infty$ is fixed by $f$, then $P_n(1,0)\neq 0$ and $Q_n(1,0)=0$, and consequently $\deg F_n=d^n$.
 \begin{lemma}\label{non_isotriviality}
 	Let $\ell\neq p$ be a prime. Suppose that $f$ is not isotrivial,
 	that $\infty$ is fixed by $f$, that $0$ is not post-critical for
 	$f$, and that $0$ is not fixed by $f$. Moreover, if $\ell\nmid d$ let $b\in\{1,\ldots,\ell-1\}$ be the integer satisfying $1+bd\equiv 0\pmod{\ell}$.
 	For every $n\geq 1$, define
 	$$
 	H_n(x)=
 	\begin{cases}
 		F_n(x),&\text{if }\ell\mid d,\\[2mm]
 		F_n(x)F_{n+1}(x)^b,&\text{if }\ell\nmid d.
 	\end{cases}
 	$$
 	Let $C_n$ be the smooth projective normalization of the affine
 	curve $y^\ell=H_n(x)$. Then for $n\gg 0$, the curve $C_n$ is
 	not isotrivial and $g(C_n)\to \infty$.
 \end{lemma}
 \begin{proof}
 	Since $p>d$, the map $f$ is separable. Since $0$ is not	post-critical, the map $f^n$ is unramified above $0$ for every	$n\geq1$. Moreover, $\infty\notin f^{-n}(0)$ because it is fixed. It follows that $F_n$ has precisely $d^n$ distinct roots, and that $F_n^{-1}(0)=f^{-n}(0)$. In particular, $0$ is not
 	exceptional for $f$.
 	
 	Notice that $\gcd(F_n,F_{n+1})=1$, since a common root $\gamma$ would satisfy $f^n(\gamma)=0=f^{n+1}(\gamma)$ and hence $f(0)=0$, contrary to the assumption that $0$ is not fixed.
 	
 	Let	$\pi_n\colon C_n\longrightarrow\mathbb P^1$	be the morphism induced by the $x$-coordinate. Every root of $F_n$ is simple in $H_n$; hence $H_n$ is not an $\ell$-th power in $\overline{F}(x)$. It follows that $C_n$ is geometrically integral and $\pi_n$ is a separable morphism of degree $\ell$.
 	
 	Since every root of $H_n$ is a branch point of $\pi_n$ and the ramification points above it all have ramification index $\ell$, it is easy to see by the Riemann--Hurwitz formula that
 	$$g(C_n)\to \infty \mbox{ as } n\to \infty.$$
 	
 	By \cite[Theorem~3.1]{hindes1}, the set $f^{-n}(0)$ is not isotrivial for $n\gg 0$. Since this set
 	is contained in the branch locus of $\pi_n$, the latter
 	is itself not isotrivial. The conclusion now follows from
 	\cite[Theorem~5.2]{hindes1}.
 \end{proof}
 
 The following lemma is the analogous of \cite[Lemma 12]{bridy} in positive characteristic.
 \begin{lemma}\label{ramification_index}
 	Suppose that $f$ is not isotrivial, that $\infty$ is fixed by $f$,
 	that $0$ is not post-critical for $f$, and that $0$ is not fixed
 	by $f$. Let $e\geq2$ be an integer which is not a power of $p$,
 	let $a\in F$ be wandering for $f$, and let $S$ be a finite set of
 	places of $F$. Then there exist $n>0$ and a place
 	$\mathfrak p\notin S$ such that
 	\[
 	v_{\mathfrak p}(f^n(a))>0
 	\qquad\text{and}\qquad
 	e\nmid v_{\mathfrak p}(f^n(a)).
 	\]
 \end{lemma}
 
 \begin{proof}

 	Now suppose by contradiction that for every $n>0$ and every
 	$\mathfrak p\notin S$ one has
 	\begin{equation}\label{eq:positive-divisibility}
 		v_{\mathfrak p}(f^n(a))>0
 		\quad\Longrightarrow\quad
 		e\mid v_{\mathfrak p}(f^n(a)).
 	\end{equation}
 	
 	Since $e$ is not a power of $p$, there exists a prime $\ell\neq p$ dividing $e$. We enlarge $S$ so that:
 	\begin{itemize}
 		\item $S$ is nonempty;
 		\item $P$ and $Q$ have coefficients in $\mathcal O_{F,S}$;
 		\item
 		$\operatorname{Res}(P,Q)\in\mathcal O_{F,S}^{\times}$;
 		\item $\mathcal O_{F,S}$ is a principal ideal domain.
 	\end{itemize}
 	
 	Write $a=A_0/B_0$ with coprime $A_0,B_0\in\mathcal O_{F,S}$ and $B_0\neq0$, and put
 	$$A_m=P_m(A_0,B_0),	\qquad 	B_m=Q_m(A_0,B_0) \qquad(m\geq0),$$
 	so that $f^m(a)=[A_m:B_m]$.
 	
 	We claim that $A_m$ and $B_m$ are coprime in $\mathcal O_{F,S}$ for every $m$. For $m=0$ this is obvious. If a place $\mathfrak p\notin S$ divided both $A_{m+1}$ and $B_{m+1}$, then	the reductions of $P$ and $Q$ modulo $\p$ would have the common projective zero $[\overline{A}_m:\overline{B}_m]$, contrary to the fact that their resultant is a unit at $\mathfrak p$. The claim therefore follows by induction.
 	
 	Thus, for every $m\geq1$ such that $A_m\neq0$, if $v_{\mathfrak p}(A_m)>0$ then $v_{\mathfrak p}(B_m)=0$, and hence
 	$$
 	v_{\mathfrak p}(A_m)
 	=
 	v_{\mathfrak p}(f^m(a)).
 	$$
 	It follows from \eqref{eq:positive-divisibility} that $\ell\mid v_{\mathfrak p}(A_m)$ for every $\mathfrak p\notin S$.
 	
 	Now notice the following: there exists a finite extension $L/F$ such that every $A_m$ is an $\ell$-th power in $L$. In fact, since the class group of $\O_S$ is finite, it is generated by a finite number of primes $\mathfrak p_1,\ldots,\mathfrak p_r$, and there are positive integers $n_1,\ldots,n_r$ such that $\mathfrak p_i^{n_i}$ is principal generated by some $\gamma_i$ for every $i$. It is then easy to show that $F'\coloneqq F(\sqrt[n_1]{\gamma_1},\ldots,\sqrt[n_r]{\gamma_r})$ has the property that every prime of $\O_S$ becomes principal in the integral closure of $\O_S'$ in $F'$. Since $\ell\mid v_\p(A_m)$ for every $\p\notin S$, the fractional ideal generated by $A_m$ in $\O_S'$ is generated by $\delta^\ell$ for some $\delta\in F'$. Hence $A_m$ and $\delta^\ell$ differ by a unit of $\O_S'$; however the group of units of $\O_S'$ has finite rank, and hence extracting $\ell$-th roots of a finite number of generators one finds a finite extension $L/F'$ with the desired property.
 	
 	Let $N$ be sufficiently large so that the curve $C_N$ supplied by
 	Lemma~\ref{non_isotriviality} is non-isotrivial and satisfies $g(C_N)\ge 2$. Since $a$ is wandering and $\infty$ is fixed, $f^m(a)\ne \infty$ for every $m$. Thus $B_m\neq0$ for every $m$. Put
 	\[
 	x_m=f^m(a)=\frac{A_m}{B_m}.
 	\]
 	The compatibility of the homogeneous iterates gives $P_j(A_m,B_m)=A_{m+j}$ for every $j,m\geq0$.
 	
 	Suppose first that $\ell\mid d$. Since $\ell\mid d^N$, define
 	\[
 	y_m=
 	\frac{z_{m+N}}{B_m^{d^N/\ell}}\in L.
 	\]
 	Then
 	\[
 	y_m^\ell
 	=
 	\frac{A_{m+N}}{B_m^{d^N}}
 	=
 	F_N(x_m)
 	=
 	H_N(x_m).
 	\]
 	
 	Now suppose that $\ell\nmid d$, and let $b$ be as in
 	Lemma~\ref{non_isotriviality}. Since $\ell\mid 1+bd$, we may define
 	\[
 	y_m=
 	\frac{z_{m+N}z_{m+N+1}^{\,b}}
 	{B_m^{d^N(1+bd)/\ell}}=\frac{A_{m+N}A_{m+N+1}^{\,b}}
 	{B_m^{d^N+bd^{N+1}}}=F_N(x_m)F_{N+1}(x_m)^b=
 	H_N(x_m)
 	\in L.
 	\]
 	
 	The points $x_m$ are pairwise distinct because $a$ is wandering.
 	Since $H_N$ has only finitely many roots, $H_N(x_m)\neq0$ for infinitely many $m$. For these $m$, the point $(x_m,y_m)$ belongs to the smooth locus of the affine curve $y^\ell=H_N(x)$, because $\ell\neq p$ and $y_m\neq0$. It therefore determines an $L$-rational point on its smooth projective normalization $C_N$.
 	We have consequently constructed infinitely many distinct points
 	in $C_N(L)$; the curve $C_N$ is not isotrivial and has genus $\ge 2$ by construction. We obtain thus a contradiction via Samuel's theorem \cite{samuel}.
 \end{proof}

\begin{proof}[Proof of Theorem \ref{finite_ramification_implies_pcf}]
    If $f$ is isotrivial, there is nothing to prove. Hence we can assume that it is not. By contradiction, assume that $f$ is not PCF. It is clear that for any $\beta\in f^{-\infty}(\alpha)$, the extension $F(\beta)_\infty(f,\beta)/F(\beta)$ ramifies at finitely many places as well. Also, the intersection of the post-critical set of $f$ with $f^{-\infty}(\alpha)$ is finite, and since $f^{-\infty}(\alpha)$ itself is infinite (since $\alpha$ is not exceptional)  there is always some $\beta\in f^{-\infty}(\alpha)$ that is not post-critical for $f$. Hence up to replacing $F$ by $F(\beta)$ we can assume that $\alpha$ is not post-critical for $f$. Moreover, notice that if $F_\infty(f,\alpha)/F$ is finitely ramified then so is $F'_\infty(f,\alpha)/F$ for any $F'$ that is a finite extension of $F$. Hence up to passing to a finite extension of $F$ and conjugating, we can assume that $\infty$ is a fixed point for $f$, that $f$ has a wandering critical point $x_0\in F$ and that $\alpha=0$.

The argument is now exactly the one used to prove \cite[Theorem 5]{bridy}, by making use of Lemma \ref{ramification_index} in place of \cite[Lemma 12]{bridy}. We provide here the only non-trivial modification. Let $e$ be the ramification index of $x_0$; since $p>d$, it must be $e<p$. Granted this, now the argument proceeds as in \cite[Theorem 5]{bridy}, verbatim, as the denominators present in that argument are all coprime to $p$.
\end{proof}

\section{Böttcher's coordinates} \label{subsection: Bottcher as formal}
This section is mostly an adaptation of the material in \cite{ingram1}. The difference in content is that we generalize that material to maps having a superattracting fixed points and the main difference in the exposition is that we rely on the Banach's fixed point's theorem (Theorem \ref{Thm: Banach} below) in establishing the existence of Böttcher's coordinates: the reader of both works will recognize that in the arguments in \cite{ingram1} one can find back the proof of the fixed point theorem applied to this special case. We recall the theorem here. A proof can be found for example in \cite[Theorem 9.23]{rudin}.
\begin{theorem} \label{Thm: Banach} \textup{(Banach's fixed point theorem)} Let $(X,\textup{dist})$ be a complete metric space. Let $f:X \to X$ be such that there exists a real number $\lambda<1$ such that for all $x,y$ in $X$ one has that
	$$\textup{dist}(f(x),f(y)) \leq  \lambda \cdot \textup{dist}(x,y).
	$$
	Then $f$ admits precisely one fixed point $b$. Furthermore the point $b$ can be obtained as the limit
	$$b=\lim_{n \to \infty} f^n(x)
	$$
	for any choice of $x$ in $X$.  
\end{theorem}
A function $f$ as in Theorem \ref{Thm: Banach} is (usually) called a \emph{contraction}. 
\subsection{Böttcher's coordinates as formal power series}
Let $R$ be a commutative ring with unit and let $e$ be in $\Z_{\geq 2}$. From now on, we'll make the fundamental assumption that $e\in R^{\times}$.

Denote by $R\llbracket x\rrbracket$ the power series ring with coefficients in $R$. We metrize $R\llbracket x\rrbracket$ with the absolute value given by
$$|f(z)|=\text{exp}(-\text{ord}(f)),
$$
where $\text{ord}(f)$ is the first non-vanishing coefficient of $f$ if $f$ is non-zero, and is $\infty$ otherwise (and we use the convention that $\text{exp}(-\infty)=0$). This absolute value induces a metric given by the formula
$$|f-g|,
$$
for each $f,g$ in $R\llbracket x\rrbracket$. We denote this metric still by $|\cdot|$, by abuse of notation. The proof of the following proposition is straightforward. 
\begin{proposition} \label{prop: x-adic metric}
	A sequence in $(R\llbracket x\rrbracket,|\cdot|)$ is Cauchy if and only if every coefficient becomes eventually constant. The resulting power series obtained by the eventual values of all coefficients is the limit. In particular $(R\llbracket x\rrbracket,|\cdot|)$ is complete. 
\end{proposition}
\begin{remark} \label{remark: definition of 1/d-roots}
	We remark that the monoid $1+x\cdot R\llbracket x\rrbracket$ is a subgroup of $R\llbracket x\rrbracket^{\times}$ where $e$-powering acts as an automorphism. In fact, it is a standard result in commutative algebra that every element is invertible. Moreover, the endomorphism 
	$$f \mapsto f^e,
	$$
	preserves each of the subgroups $1+x^i\cdot R\llbracket x\rrbracket$, for $i \geq 1$. The induced endomorphism on
	$$\frac{1+x^i\cdot R\llbracket x\rrbracket}{1+x^{i+1}\cdot R\llbracket x\rrbracket}
	$$
	is simply multiplication by $e$, which is invertible. In particular, recalling that $R\llbracket x\rrbracket$ is complete, one deduces that the $e$-powering endomorphism of $1+x\cdot R\llbracket x\rrbracket$ is an automorphism. This is an elementary general property of filtered endomorphisms of complete filtered groups: they are automorphisms if and only if they are on every successive quotient (see \cite[Proposition 3.6]{pagano}). 
	
	We denote by $f^{\frac{1}{e}}$ the inverse automorphism, hence so far defined on $1+x \cdot R\llbracket x\rrbracket$. We extend the notation $f^{\frac{1}{e}}$ for 
	$$f=x^{e\cdot s} \cdot (1+x\cdot g(x)),
	$$
	with $s \in \Z_{\geq 0}$, by declaring
	$$f^{\frac{1}{e}}:=x^s \cdot (1+x\cdot g(x))^{\frac{1}{e}}. 
	$$
	Observe that we can further extend $(-)^{\frac{1}{e}}$ to every element $f$ of $R\llbracket x\rrbracket[x^{-1}]$ of the form
	$$f=x^{e\cdot s} \cdot (1+x\cdot g(x)),
	$$
	for $s$ in $\Z$ and $g(x)\in R\llbracket x\rrbracket$ by the same formula. 
\end{remark}
We will now do a convenient notational switch and swap the roles of $x$ and $x^{-1}$ and consider the ring $R\llbracket \frac{1}{x}\rrbracket[x]$ (which can be intuitively thought as the ring of power series around $\infty$). Formally this can be defined by starting from the ring of Laurent polynomials $R[x,x^{-1}]$, which can be completed with respect to the absolute value corresponding to $-\text{ord}$. 

Of course we have an isomorphism between $R\llbracket x\rrbracket[x^{-1}]$ and $R\llbracket \frac{1}{x}\rrbracket[x]$, obtained simply by swapping $x$ and $x^{-1}$. We can transport along this isomorphism the map $(-)^{\frac{1}{d}}$ defined in Remark \ref{remark: definition of 1/d-roots} and thus obtain a map
$$(-)^{\frac{1}{e}}\colon \bigcup_{s \in \Z}x^{e \cdot s} \cdot \left(1+\frac{1}{x}\cdot R\left\llbracket \frac{1}{x}\right\rrbracket\right) \to \bigcup_{s \in \Z}x^{s} \cdot \left(1+\frac{1}{x} \cdot R\left\llbracket \frac{1}{x}\right\rrbracket\right).
$$
\begin{remark} \label{rmk: 1/e is a bijection}
	The map defined right above is the compositional inverse of $e$-powering on the set 
	$$\bigcup_{s \in \Z}x^{e \cdot s} \cdot \left(1+\frac{1}{x}\cdot R\left\llbracket \frac{1}{x}\right\rrbracket\right).
	$$
\end{remark}
\begin{definition}
	We let $\mathcal{B}$ be the $R\llbracket \frac{1}{x}\rrbracket$-coset of the element $x$, that is, we put
	$$\mathcal{B}\coloneqq x+R\left\llbracket \frac{1}{x}\right\rrbracket.
	$$
\end{definition}
Observe that through translation by $x$ the coset $\mathcal{B}$ inherits the structure of complete metric space from $R\llbracket \frac{1}{x}\rrbracket$ thanks to Proposition \ref{prop: x-adic metric}.

We now fix an element $f(x)$ in $R\llbracket \frac{1}{x}\rrbracket[x]$ of the form
$$f\coloneqq x^e+\sum_{i=0}^{e-1}a_i \cdot x^i+\sum_{j=0}^{\infty}\frac{c_j}{x^j}. 
$$
Below we shall suggestively refer to this shape for $f$ as to have a pole of order $e$ at $\infty$ (which, by definition, we demand to include the monicity assumption on $x^e$). 
\begin{remark} \label{remark: composing with f is well defined on B}
	For $g$ in $\mathcal{B}$, the expression $g \circ f$ is well defined. In fact, let
	$$g\coloneqq x+\sum_{i=0}^{\infty}\frac{b_i}{x^i}.
	$$
	Then formally taking
	$$g(f(x))\coloneqq f(x)+\sum_{i=0}^{\infty}\frac{b_i}{x^{e \cdot i}(1+\sum_{j=1}^{\infty}\frac{\gamma_j}{x^j})^i}
	$$
	and recalling that for every $i$ the inverse of $\left(1+\sum_{j=1}^{\infty}\frac{\gamma_j}{x^j}\right)^i$ lies in $1+R\llbracket \frac{1}{x}\rrbracket$, we see that each degree is touched finitely many times and we get a series that converges $\frac{1}{x}$-adically; moreover
	$$g(f(x)) \in x^e \cdot \left(1+\frac{1}{x} \cdot R\left\llbracket \frac{1}{x}\right\rrbracket\right). 
	$$
	Hence we can associate to $g$ the element
	$$(g(f(x)))^{\frac{1}{e}}.
	$$
	We denote this element as $\phi_{e,f}(g)$.
\end{remark}
The next proposition is the key that allows to prove the existence of local Böttcher's coordinates. 
\begin{proposition} \label{prop: phi is a contraction}
	We have that $\phi_{e,f}\colon \mathcal B \to \mathcal B$ is a contraction. 
	\begin{proof}
		
		Let $g_1,g_2$ be two distinct elements of $\mathcal{B}$. Let us write
		$$g_1(x)=x+b_0+\sum_{i=1}^{\infty}\frac{b_{-i}}{x^i},
		$$
		and
		$$g_2(x)=x+c_0+\sum_{i=1}^{\infty}\frac{c_{-i}}{x^i}.
		$$
		We write
		$$(g_1 \circ f)(x)=f(x)+b_0+\sum_{i=1}^{\infty}\frac{b_{-i}}{(f(x))^i}
		$$
		and
		$$(g_2\circ f)(x)=f(x)+c_0+\sum_{i=1}^{\infty}\frac{c_{-i}}{(f(x))^i}.
		$$
		To compute these we write
		$$f(x)=x^e+a_{e-1}x^{e-1}+\ldots+a_0+\sum_{i=1}^{\infty}\frac{a_{-i}}{x^i}
		$$
		and expand as
		$$(g_1 \circ f)(x)=x^e \cdot \left(1+\sum_{j=1}^{\infty}\frac{a_{e-j}}{x^j}+\frac{b_0}{x^e}+\sum_{i=1}^{\infty}\frac{b_{-i}}{x^{e(i+1)} \cdot \left(1+\sum_{j=1}^{\infty}\frac{a_{e-j}}{x^j}\right)^i}\right),
		$$
		and
		$$(g_2 \circ f)(x)=x^e \cdot \left(1+\sum_{j=1}^{\infty}\frac{a_{e-j}}{x^j}+\frac{c_0}{x^e}+\sum_{i=1}^{\infty}\frac{c_{-i}}{x^{e(i+1)} \cdot \left(1+\sum_{j=1}^{\infty}\frac{a_{e-j}}{x^j}\right)^i}\right).
		$$
		Hence we conclude that
		$$\phi_{e,f}(g_1)=x \cdot \left(1+\sum_{j=1}^{\infty}\frac{a_{e-j}}{x^j}+\frac{b_0}{x^e}+\sum_{i=1}^{\infty}\frac{b_{-i}}{x^{e(i+1)} \cdot \left(1+\sum_{j=1}^{\infty}\frac{a_{e-j}}{x^j}\right)^i}\right)^{\frac{1}{e}}
		$$
		and
		$$\phi_{e,f}(g_2)=x \cdot \left(1+\sum_{j=1}^{\infty}\frac{a_{e-j}}{x^j}+\frac{c_0}{x^e}+\sum_{i=1}^{\infty}\frac{c_{-i}}{x^{e(i+1)} \cdot \left(1+\sum_{j=1}^{\infty}\frac{a_{e-j}}{x^j}\right)^i}\right)^{\frac{1}{e}}.
		$$
		Now let $i_0$ be the smallest nonnegative integer such that $b_{i_0} \neq c_{i_0}$ (recall that $g_1 \neq g_2$). Then we have that 
		$$\text{dist}(g_1,g_2)=\text{exp}(-i_0).
		$$
		Visibly, the two expressions 
		$$1+\sum_{j=1}^{\infty}\frac{a_{e-j}}{x^j}+\frac{b_0}{x^e}+\sum_{i=1}^{\infty}\frac{b_{-i}}{x^{e(i+1)} \cdot \left(1+\sum_{j=1}^{\infty}\frac{a_{e-j}}{x^j}\right)^i}
		$$
		and
		$$1+\sum_{j=1}^{\infty}\frac{a_{e-j}}{x^j}+\frac{c_0}{x^e}+\sum_{i=1}^{\infty}\frac{c_{-i}}{x^{e(i+1)} \cdot \left(1+\sum_{j=1}^{\infty}\frac{a_{e-j}}{x^j}\right)^i},
		$$
		agree up to order $e \cdot (i_0+1)$. Since $(-)^{\frac{1}{e}}$ is a filtration preserving isomorphism, we find that $\phi_{e,f}(g_1)$ and $\phi_{e,f}(g_2)$ will agree up to order $e \cdot (i_0+1)-1=e \cdot i_0+e-1$, that is larger than $i_0+1$ because $e\ge 2$ (the $-1$ is caused by multiplication by $x$ that shifts the expression down by $1$). 
		
		Therefore, we deduce that
		$$\text{dist}(\phi_{e,f}(g_1),\phi_{e,f}(g_2)) \leq \text{exp}(-i_0) \cdot \text{exp}(-1)=\text{exp}(-1) \cdot \text{dist}(g_1,g_2).
		$$
		This shows that the map $\phi_{e,f}$ is a contraction as desired. 
		\end{proof}
	\end{proposition}
We now derive the following corollary.
\begin{corollary} \label{cor: existence of B coordinates, formal}
Let $f$ be in $R\llbracket \frac{1}{x}\rrbracket[x]$ with a pole of order $e$ at infinity. Then there exists a unique $b$ in $\mathcal{B}$ such that
$$b \circ f=b^e.
$$
Furthermore $b$ can be written as the $\frac{1}{x}$-adic limit
$$b=\lim_{n \to \infty}(f^n)^{\frac{1}{e^n}}.
$$
\begin{proof}
Thanks to Proposition \ref{prop: phi is a contraction} we conclude via Theorem \ref{Thm: Banach} that the map $\phi_{e,f}$ must have precisely one fixed point $b$ in $\mathcal{B}$. The equation
$$\phi_{e,f}(b)=b,
$$
can be rewritten as
$$(b \circ f)^{\frac{1}{e}}=b.
$$
This equation certainly implies 
$$b \circ f=b^e,
$$
establishing the existence part of the desired statement. Now let $b_0$ be an element of $\mathcal B$ such that
$$b_0 \circ f=b_0^{e}.
$$
We can apply on both sides the map $(-)^{\frac{1}{e}}$. In view of Remark \ref{rmk: 1/e is a bijection}, this yields 
$$\phi_{e,f}(b_0)=b_0.
$$
Hence $b_0$ is a fixed point for $\phi_{e,f}$ and therefore $b_0=b$, in view of the uniqueness part of Theorem \ref{Thm: Banach} combined, once more, with Proposition \ref{prop: phi is a contraction}. 

Finally, taking as starting point of the metric space the element $x$ of $\mathcal{B}$, we see that
$$\phi_{e,f}^{n}(x)
$$
converges $\frac{1}{x}$-adically to the fixed point $b$, in view of Theorem \ref{Thm: Banach} combined with Proposition \ref{prop: phi is a contraction}. Since
$$\phi_{e,f}(f^{n}(x)^{\frac{1}{e^n}})=((f^{n}(f(x)))^{\frac{1}{e^n}})^{\frac{1}{e}}=(f^{n+1}(x))^{\frac{1}{e^{n+1}}},
$$
we obtain, by induction, that
$$\phi_{e,f}^{n}(x)=(f^n(x))^{\frac{1}{e^n}}
$$
and hence $b$ is the $\frac{1}{x}$-adic limit of 
$$\{(f^n(x))^{\frac{1}{e^n}}\}_{n \geq 1}. 
$$

\end{proof}
\end{corollary}
Suppose now furthermore that $R$ is an integral domain with fraction field $Q(R)$, and let 
$$\phi\coloneqq\frac{f_1(x)}{f_2(x)} \in Q(R)(x)
$$
be a rational function with the property that both $f_1(x)$ and $f_2(x)$ are monic polynomials in $R[x]$ such that 
$$\text{deg}(f_1)-\text{deg}(f_2)=e\ge 2.
$$
Observe that through geometric expansion we see that $\phi$ can be naturally viewed as an element of $R\llbracket \frac{1}{x}\rrbracket[x]$ having a pole of order $e$ (in the sense explained above). In particular Corollary \ref{cor: existence of B coordinates, formal} gives us a well defined element
$$b_{\phi} \in \mathcal{B},
$$
such that 
$$b_{\phi} \circ \phi=b_{\phi}^e.
$$
We now want to introduce the compositional inverse of $b_{\phi}$. Observe that the set $\mathcal{B}$ is a submonoid of $(R\llbracket \frac{1}{x}\rrbracket,\circ)$. 
\begin{proposition}\label{inverse_of_b}
The submonoid $\mathcal{B}$ is a subgroup. That is: every element of $\mathcal{B}$ admits a compositional inverse.
\begin{proof}
It suffices to show that every element of the submonoid $\mathcal{B}$ admits a left inverse. Let 
$$g_1=x+a_0+\sum_{i=1}^{\infty}\frac{a_i}{x^i}.
$$
We show inductively the existence of a left inverse $g_2=x+b_0+\sum_{i=1}^{\infty}\frac{b_i}{x^i}$, by showing for each $n$ in $\Z_{\geq 0}$ 
$$\left(x+\sum_{i=0}^{n}\frac{b_i}{x^i}\right) \circ g_1=x+O\left(\frac{1}{x^{n+1}} \right)=x+\frac{\lambda}{x^{n+1}}+O\left(\frac{1}{x^{n+2}} \right).
$$

For $i=0$, we put $b_0=-a_0$. We see that with this choice we have that $g_2 \circ g_1=x+O\left(\frac{1}{x} \right)$ as desired. Suppose we have found $b_0,\ldots,b_n$ that verify the claim; we now show that we can find some $b_{n+1}$ that extends the construction. We have that
$$\left(x+\sum_{i=0}^{n+1}\frac{b_i}{x^i}\right) \circ g_1=x+\frac{\lambda}{x^{n+1}}+\frac{b_{n+1}}{g_1(x)^{n+1}}+O\left(\frac{1}{x^{n+2}} \right).
$$
Using that $g_1(x)$ is in $\mathcal{B}$, we have that 
$$\frac{b_{n+1}}{g_1(x)^{n+1}}=\frac{b_{n+1}}{x^{n+1}}+O\left(\frac{1}{x^{n+2}} \right).
$$
All in all we have that 
$$\left(x+\sum_{i=0}^{n+1}\frac{b_i}{x^i}\right) \circ g_1=x+\frac{b_{n+1}+\lambda}{x^{n+1}}+O\left(\frac{1}{x^{n+2}} \right).
$$
Hence we see that $b_{n+1}=-\lambda$ gives the desired conclusion. 
\end{proof}
\end{proposition}
In this way for $\phi$ as above, we have a well defined $b_{\phi}^{-1}$ in $\mathcal{B}$. 
\subsection{Böttcher's coordinates as analytic objects} \label{subsection: Bottcher as analytic function}
The purpose of this section is to view $b_{\phi}$ at the end of Section \ref{subsection: Bottcher as formal} as an analytic function: ultimately we want to be able to evaluate it in some regions of a tree. 

Let $(K,|\cdot|)$ be a non-archimedean local field, i.e.\ a valued field which is either a finite extension of $\Q_p$ for some prime number $p$ (with the unique extension of the $p$-adic absolute value) or isomorphic to the fraction field of $\mathbb{F}_q\llbracket T\rrbracket$ (with the standard absolute value), i.e.\ $\mathbb{F}_q\llbracket T\rrbracket[T^{-1}]$, where $q=p^f$ for some positive integer $f$. Let $e$ be in $\Z_{\geq 2}$ an integer smaller than $p$. We fix once and for all a separable closure of $K$ and denote it by $K^{\text{sep}}$. We let $G_K\coloneqq\text{Gal}(K^{\text{sep}}/K)$. Recall that the absolute value $| \cdot |$ extends uniquely to $K^{\text{sep}}$. 

Let us now specialize the discussion of Section \ref{subsection: Bottcher as formal} to the case $R=\mathcal{O}_K$. Let $\phi\coloneqq\frac{f_1}{f_2}$ be a rational function in $K(x)$ written as the ratio of two monic polynomials $f_1,f_2$ in $\mathcal{O}_K[x]$ such that $\text{deg}(f_1)-\text{deg}(f_2)=e$. We call such a rational map a superattracting integral map of pole degree $e$. 

Observe that $b_{\phi}$ lives naturally in $x+\mathcal{O}_K\llbracket \frac{1}{x}\rrbracket$. As such it can be clearly evaluated in every element $\gamma$ of $K^{\text{sep}}$ with $|\gamma|>1$. Moreover, the functional equation
$$b_{\phi} \circ \phi=b_{\phi}^e
$$
can be as well evaluated in such elements of $K^{\text{sep}}$, yielding:
$$b_{\phi}(\phi(\gamma))=(b_{\phi}(\gamma))^e.
$$
Proposition \ref{inverse_of_b} implies that we can also evaluate the functional equation
$$b_{\phi} \circ b_{\phi}^{-1}=\text{id}=b_{\phi}^{-1} \circ b_{\phi}
$$
on every element $\gamma$ of $K^{\text{sep}}$ with $|\gamma|>1$. Observe that $b_{\phi}$ clearly commutes with the action of $G_K$, since this latter action is by isometries (by the uniqueness of the extension of $|\cdot|$ to $K^{\text{sep}}$). All in all, the map $b_\phi$ defines a $G_K$-equivariant bijection
$$b_{\phi}\colon \{\gamma \in K^{\text{sep}}: |\gamma|>1 \} \righttoleftarrow.
$$
One can upgrade as well the functional equation
$$b_{\phi}=\lim_{n \to \infty}(\phi^n)^{\frac{1}{e^n}}
$$
into an equation of analytic function and the application of Theorem \ref{Thm: Banach} can be used to give finer convergence properties than those given in Theorem \ref{cor: existence of B coordinates, formal}. As we shall not need this here, we will not give further details and refer the interested reader to \cite{ingram1} for the polynomial case. 
\subsection{Root extraction for non-constant basepoints} \label{subsection: root extraction}
We begin with an elementary fact that will be used to recognize arboreal subrepresentations later in this section. 
\begin{proposition} \label{prop: criterion to be a subtree}
	Let $e \leq d$ be positive integers. Let $\mathcal{G}$ be a subgraph of an infinite $d$-regular rooted tree $T$ with the following three properties:
	\begin{itemize}
		\item the root of $T$ is a vertex of $\mathcal{G}$;
		\item if a vertex is in $\mathcal{G}$ then so is its unique ascendant, and they are linked in $\mathcal{G}$;
		\item every vertex of $\mathcal{G}$ has precisely $e$ descendants in $\mathcal{G}$.
	\end{itemize}
	Then $\mathcal{G}$ is isomorphic to an infinite $e$-regular tree rooted at the same root of $T$.
	\begin{proof}
		We proceed by induction on the level of the tree. Clearly, by the first property we have that the part of $\mathcal{G}$ of level $0$ consists of the root of $T$. Now suppose that for a nonnegative integer $N$, we know that the part of $\mathcal{G}$ consisting of vertices at distance at most $N$ from the root, is the truncation at level $N$ of an $e$-regular tree rooted at the same root of $T$. Thanks to the second property we know that if a vertex is present at level $N+1$ then it must be linked in $\mathcal{G}$ to its unique ascendant at level $N$. And by the third property we know that each vertex at level $N$ has precisely $e$ descendants in $\mathcal{G}$. Hence the part of $\mathcal{G}$ at level $N+1$ consists of a choice of $e$ descendants for each of the elements of $\mathcal{G}$ at level $N$. Therefore $\mathcal{G}$ at level $N+1$ is the truncation at level $N+1$ of a $e$-regular tree rooted at the root of $T$. 
		
		This completes the induction and shows that this property holds at all levels, which is readily seen to be equivalent to the desired conclusion. 
	\end{proof}
\end{proposition}

We recall the notation of Section \ref{subsection: Bottcher as analytic function}, that we will use in the rest of the section: $(K, |\cdot|)$ is a non-archimedean local field, $e\ge 2$ is an integer that is invertible in $K$, and $\phi\in K(x)$ is a rational function that can be written as $f_1/f_2$ with $f_1,f_2\in \O_K[x]$ and monic, and such that $\deg f_1-\deg f_2=e$.

Thanks to the following proposition we are able to produce, for base points $\alpha$ with $|\alpha|>1$, a natural Galois invariant subtree of degree $e$ inside $T_\infty(\phi,\alpha)$. Notice that this is not obvious, as $T_\infty(\phi,\alpha)$ has degree $\text{deg}(\phi)$, which is different from $e$ precisely if $\phi$ is not a polynomial. Moreover, we show that this tree is isomorphic, as a $G_K$-set, to a tree of preimages of the map $x^e$.
\begin{proposition} \label{prop: canonica degree e subtree}
	Assume the above notation. Let $\alpha$ be in $K$ with $|\alpha|>1$.
	\begin{enumerate}
		\item Let $\gamma\in K$ be such that $|\gamma|>1$. The map $b_\phi$ induces an isomorphism of $G_K$-sets
		$$\{\beta\in K^{\text{sep}}\colon |\beta|>1,\,\, \phi(\beta)=\gamma\}\to \{\beta\in K^{\text{sep}}\colon \beta^e=b_\phi(\gamma)\}.$$
		\item The subset of vertices
		$$T_{\infty}^e(\phi,\alpha)\coloneqq T_{\infty}(\phi,\alpha) \cap \{\beta \in K^{\textup{sep}}\colon|\beta|>1\},
		$$
		with the induced graph structure is a regular tree of degree $e$ rooted at $\alpha$. Furthermore, such tree is preserved by $G_K$.
		\item $b_{\phi}$ induces an isomorphism of $G_K$-trees
		$$T_{\infty}^e(\phi,\alpha) \to T_{\infty}(x^e,b_\phi(\alpha))
		$$
		and yields an equality of fields
		$$K(T_{\infty}^e(\phi,\alpha))=K(T_{\infty}(x^e,b_{\phi}(\alpha))).
		$$
	\end{enumerate} 
	\begin{proof}
		$(1)$ Let $\beta\in K^{\text{sep}}$ be such that $|\beta|>1$ and $\phi(\beta)=\gamma$. Then $b_\phi(\phi(\beta))=b_\phi(\gamma)$, so that $(b_\phi(\beta))^e=b_\phi(\gamma)$ (note that evaluating $b_\phi$ at $\beta$ makes sense because $|\beta|>1$). On the other hand if $\beta^e=b_\phi(\gamma)$ then $b_\phi^{-1}(\beta^e)=\gamma$, since $b_\phi$ is a bijection on the set of elements of absolute value greater than $1$. For the same reason, $\beta=b_\phi(\beta')$ for some $\beta'\in K^{\text{sep}}$ with $|\beta'|>1$. It follows that $\phi(\beta')=\gamma$. Since both $b_\phi$ and $b_\phi^{-1}$ are injective, we have found injective maps in opposite directions between finite sets; this means that such maps are bijections. The fact that $b_\phi$ is $G_K$-equivariant follows immediately from the fact that $G_K$ acts isometrically on $K^{\text{sep}}$.
		
		$(2)$ First of all, by hypothesis $\alpha\in T^e_\infty(\phi,\alpha)$. Next, suppose that $\gamma\in T^e_\infty(\phi,\alpha)$. Then in particular $|\gamma|>1$, and it is immediate to see that $|\phi(\gamma)|>1$ as well. Therefore $\phi(\gamma)$, that is the unique ascendant of $\gamma$ in $T_\infty(\phi,\alpha)$, belongs to $T^e_\infty(\phi,\alpha)$. Finally, let $\gamma\in T^e_\infty(\phi,\alpha)$. By point $(1)$, the set of $\beta\in K^{\text{sep}}$ with $|\beta|>1$ and such that $\phi(\beta)=\gamma$ has precisely $e$ elements, since $e$ is invertible in $K$. It follows that $\gamma$ has exactly $e$ descendants. By Proposition \ref{prop: criterion to be a subtree}, it follows that $T_\infty^e(\phi,\alpha)$ is a regular rooted tree of degree $e$.
		
		$(3)$ By point $(1)$ we get that $b_\phi$ induces a level-preserving $G_K$-equivariant bijection $T_\infty^e(\phi,\alpha)\to T_\infty(x^e,b_\phi(\alpha))$. Since $b_\phi\circ \phi=b_\phi^e$ on $T_\infty^e(\phi,\alpha)$, such bijection is in fact an isomorphism of $G_K$-trees. Finally, if $\gamma \in T_\infty^e(\phi,\alpha)$ then $b_\phi(\gamma)\in K(\gamma)$, because $K(\gamma)$ is complete and $b_\phi(\gamma)$ is given by a convergent series of elements of $K(\gamma)$. Therefore, $K(T_{\infty}^e(\phi,\alpha))\supseteq K(T_{\infty}(x^e,b_{\phi}(\alpha)))$. The reverse inclusion is proven by the same argument on the map $b_\phi^{-1}$.
	\end{proof}
\end{proposition}

Proposition \ref{prop: canonica degree e subtree} immediately implies the following corollary. 
\begin{corollary}\label{basepoint_is_constant}
	Let $\phi$ be a superattracting integral map of pole degree $e$ defined over $K$. Let $\alpha$ be in $K$ with $|\alpha|>1$. Then the Galois group
	$$\textup{Gal}(K_{\infty}(\phi,\alpha)/K)
	$$
	is non-abelian. 
\end{corollary}
\begin{proof}
	By Proposition \ref{prop: canonica degree e subtree}, it is enough to show that $K(T_\infty(x^e,b_\phi(\alpha)))/K$ is non-abelian. This can be viewed in (at least) two different ways. First, $K(T_\infty(x^e,b_\phi(\alpha)))/K$ has unbounded prime-to-$p$ ramification, while the tame inertia in an abelian local extension is bounded (see also \cite[Lemma 3.11]{ferpag1}. Secondly, in a more direct way, one can infer the non-abelianity of $K(T_\infty(x^e,b_\phi(\alpha)))/K$ from \cite[Theorem VI.9.4]{lang}.
\end{proof}

\section{PCF implies isotrivial} \label{Sct: Pcf impl isotriv}
For the rest of this section $d \geq e$ will be two integers in $\Z_{\geq 2}$. For a field $K$, we denote by 
$$X_{d,e}(K),
$$
the set of rational maps $f$ in $K(t)$, which satisfies the following
$$f=\frac{f_1(t)}{f_2(t)},
$$
where $f_1,f_2$ are monic coprime polynomial in $K[t]$ of degree respectively $d$ and $d-e$ and with $f_1(0)=0$. We define
$$X_{d,e}^{\textup{PCF}}(K),
$$
to be the subset of $X_{d,e}(K)$ consisting of PCF maps. The main result of this section is the following. Informally it states, over fields of characteristic $0$ or larger than $d$, it is not possible to find $1$-parameter families of PCF maps of degree $d$ with infinity superattracting of order $e$. For a proof see \cite[Section 3]{ingram2}, taking into account that our normalized form in $X_{d,e}$ becomes his upon conjugating by $z \mapsto \frac{1}{z}$. 
\begin{theorem} \label{thm: Thurston rig. fct field version}
Let $d,e$ be as above. Let $k$ be a field such that $\textup{char}(k)=0$ or $\textup{char}(k)>d$. Let $K$ be a field of transcendence degree $1$ over $k$, and suppose that $k$ is algebraically closed in $K$. Then
$$X_{d,e}^{\textup{PCF}}(K)=X_{d,e}^{\textup{PCF}}(k). 
$$
\end{theorem}
We have the following trivial consequence. 
\begin{corollary}\label{isotrivial_conjugacy}
Let $d,e$ be as above. Let $k$ be a field such that $\textup{char}(k)=0$ or $\textup{char}(k)>d$. Let $K$ be a field of transcendence degree $1$ over $k$, and suppose that $k$ is algebraically closed in $K$. Let $f$ be in $K(t)$ a PCF map of degree $d$ with a fixed point of order $e$. Then there exists a finite extension $L/K$ and an element $\gamma$ in $\textup{PGL}_2(L)$ such that 
$$\gamma \circ f \circ \gamma^{-1}
$$
is in $k(t)$. 
\end{corollary}
We conclude mentioning the following consequence, which follows rapidly from the above theorems, as any positive dimensional variety needs to contain a non-empty open of some curve. 
\begin{theorem} \label{thm: Thurston's rigidity}
Let $d,e$ be as above. Let $k$ be a field such that $\textup{char}(k)=0$ or $p\coloneqq \textup{char}(k)>d$. Then
$$X_{d,e}^{\textup{PCF}}(k)=X_{d,e}^{\textup{PCF}}(k\cap \widebar{\Q})
$$
in the first case and 
$$X_{d,e}^{\textup{PCF}}(k)=X_{d,e}^{\textup{PCF}}(k\cap \widebar{\mathbb{F}}_p)
$$
in the second case. 

Furthermore, for each positive integer $N$, there are only finitely many elements in $X_{d,e}^{\textup{PCF}}(k)$ with all of the critical orbits having length bounded by $N$. 
\end{theorem}
\begin{remark} \label{Remark: comparison with Levy}
Theorem \ref{thm: Thurston's rigidity} has been announced with a different proof as well in unpublished work of Levy. 

At present, we are not able to verify the proof given in that work. In particular, it uses two times and in a crucial way that the coordinate ring of the space of rational maps as in $X_{d,e}$ equipped with an ordered list of critical points and critical values is integral over the polynomial algebra of critical values. The only evidence provided in Levy for this integrality is given by the fact that the resulting map on $\text{Spec}_{\text{max}}$ has finite fibers. Unfortunately that is not sufficient, as one can see for example with $k[x,\frac{1}{x-1}]$ as $k[x^2]$-algebra where $\textup{char}(k)\neq 2$, where the map is even surjective and with finite fibers, but the former ring is not integral on the latter. We believe it would be interesting to fill this hole in the literature and provide a proof of this integrality claim. 
\end{remark}

\section{Proof of Theorem \ref{main_thm}} \label{Section: Proof}

We recall that $F$ is a global function field of characteristic $p$, $f\in F(x)$ is a rational function of degree $d\ge 2$ such that $p>d$ and having a superattracting fixed point, and $\alpha\in F$ is non-exceptional for $f$. We begin by proving the following claim:

\vspace{3mm}

\begin{center}
$(\heartsuit)$ Let $(g,\beta)$ be a constant pair (i.e.\ a pair such that $g\in \overline{\F}_p(x)$ and $\beta\in \overline{\F}_p$) and suppose that $(f,\alpha)$ and $(g,\beta)$ are conjugate over some extension $L/F$. Then $G_\infty(f,\alpha)$ is abelian if and only if $(f,\alpha)$ and $(g,\beta)$ are conjugate over $F^{\ab}$.
\end{center}

\vspace{3mm}

Let $m\in \text{PGL}_2(L)$ be such that $m^{-1}\circ f\circ m=g$ and $m^{-1}(\alpha)=\beta$. For every $n\ge 0$, the map $m$ induces a bijection $g^{-n}(\beta)\to f^{-n}(\alpha)$ and, in turn, a bijection $g^{-\infty}(\beta)\to f^{-\infty}(\alpha)$. If $L\subseteq F^{\ab}$ then the fact that $f^{-\infty}(\alpha)\subseteq F^{\ab}$ follows immediately from the surjectivity of $m$ together with the fact that $g^{-\infty}(\beta)\subseteq \overline{\F}_p$. Conversely, suppose that $G_\infty(f,\alpha)$ is abelian and let $\gamma_1,\gamma_2,\gamma_3$ be three pairwise distinct points in $g^{-\infty}(\beta)$, which surely exist since $\beta$ is not exceptional for $g$. Then $m(\gamma_i)\in F^{\text{ab}}$ for every $i\in \{1,2,3\}$. On the other hand also $\gamma_i\in \overline{\F}_p\subseteq F^{\text{ab}}$ for every $i\in \{1,2,3\}$. The map $m$ is uniquely determined by its action on three points, and hence it needs to be defined over $F^{\text{ab}}$.

Claim $(\heartsuit)$ implies immediately one direction of Theorem \ref{main_thm}: if $(f,\alpha)$ is $F^{\text{ab}}$-conjugate to a constant pair, then $G_\infty(f,\alpha)$ is abelian.

Conversely, suppose that $G_\infty(f,\alpha)$ is abelian. By Theorem \ref{abelian_implies_fr} the extension $F_\infty^{\text{ab}}(f,\alpha)/F$ ramifies only at finitely many places, but saying that $G_\infty(f,\alpha)$ is abelian is equivalent to saying that $F_\infty^{\text{ab}}(f,\alpha)=F_\infty(f,\alpha)$. Hence by Theorem \ref{finite_ramification_implies_pcf} it follows that $f$ is PCF. Since $f$ has a superattracting fixed point, it follows by Corollary \ref{isotrivial_conjugacy} that there exists a finite extension $L/F$ such that $(f,\alpha)$ is $L$-conjugate to $(g,\beta)$ where $g\in \overline{\F}_p(x)$ has a pole of degree $e\ge 2$ at $\infty$ and $\beta\in L$. On the other hand since $\gal(F_\infty(f,\alpha)/L)$ is a subgroup of $G_\infty(f,\alpha)$, it is itself abelian, and because $(f,\alpha)$ and $(g,\beta)$ are $L$-conjugate it follows that $G_\infty(g,\beta)$ is abelian (as conjugacy induces an isomorphism of the latter with $\gal(F_\infty(f,\alpha)/L)$). Now by Corollary \ref{basepoint_is_constant} we must have $|\beta|\le 1$ for every absolute value in $L$, which amounts to say that $\beta\in \overline{\F}_p$. Finally, claim $(\heartsuit)$ shows that $(f,\alpha)$ and $(g,\beta)$ are $F^{\ab}$-conjugate, completing the proof.

\section{Abelian dynamical Galois groups and Drinfeld modules} \label{Section: Drinfeld}

One of the questions left open by Theorem \ref{main_thm} is what happens when $p\le d$. Certainly the theorem does not hold true in this setting, because Drinfeld modules with complex multiplication yield examples of non-isotrivial rational maps with virtually abelian dynamical Galois group. The map $f=Tx+x^q\in \F_q(T)[x]$ is the prototypical example of this phenomenon: it has a superattracting critical point, since it is a polynomial, it is not isotrivial, but $G_\infty(f,0)$ is abelian. Notice that it is still a PCF map, since it has no finite critical point. This section is dedicated to explaining which class of examples this map belongs to, and formulating a positive characteristic version of Conjecture \ref{ap_conjecture}.

We start by briefly recalling some definitions and facts in the theory of Drinfeld modules. Let $X$ be a smooth, projective, irreducible curve over $\F_q$ and let $\infty$ be a point of $X$. Let $A$ be the ring of functions on $X$ that are regular outside $\infty$. Let $F$ be a finitely generated field over $\F_q$. We let $F\{\tau\}$ be the ring of \emph{twisted polynomials}; its elements are polynomials in $\tau$ with the usual addition and multiplication given by the rules $\tau^{i}\tau^j=\tau^{i+j}$ for every $i,j\ge 0$ and $\tau a=a^q\tau$ for every $a\in F$. This is isomorphic to the ring $\{\sum_{i=0}^na_ix^{q^i}\colon a_i\in F, n\ge 0\}$ with the operations being addition and composition.

Now let $\gamma\colon A \to F$ be a non-constant homomorphism. A \emph{Drinfeld module} of rank $r\in \N\setminus \{0\}$ defined over $F$ is an injective homomorphism $\psi\colon A\to F\{\tau\}$ such that for every $a\in A\setminus \{0\}$ we have
$$\psi_a=\gamma(a)+\sum_{i=1}^{r\deg a}a_i\tau^i$$
for some $a_1,\ldots,a_{r\deg a}\in F$ (note that traditionally, the image of $a$ is denoted by $\psi_a$ rather than $\psi(a)$). We say that $\psi$ has \emph{generic characteristic} if $\gamma$ is injective, and \emph{special characteristic} otherwise. In the latter case, the prime ideal $\ker \gamma$ is called the \emph{characteristic} of $\psi$.

The \emph{endomorphism ring} of $\psi$ is 
$$\en(\psi)\coloneqq\{\mu\in \overline{F}\{\tau\}\colon \mu\psi(a)=\psi(a)\mu \mbox{ for every }a \in A\},$$
and it is a finitely generated projective $A$-module. If $K$ is the quotient field of $A$, we have that $\en(\psi)^0\coloneqq\en(\psi)\otimes K$ is a division algebra, and if $Z$ is its center then $\dim_Z \en(\psi)^0=d^2$ and $[Z\colon K]=e$ for some positive integers $d,e$ such that $de\mid r$. If $\psi$ has generic characteristic then $\en(\psi)^0$ is commutative, and hence $\dim_K\en(\psi)^0\mid r$. If equality holds, we say that $\psi$ has \emph{complex multiplication}.

Finally, if $\psi'\colon A\to F\{\tau\}$ is another Drinfeld module then an \emph{isogeny} is a non-zero element $\mu\in \overline{F}\{\tau\}$ such that $\mu\psi(a)=\psi'(a)\mu$ for every $a\in A$. An isogeny of $\tau$-degree $0$ is called an \emph{isomorphism}. If $\psi$ has special characteristic and it is isomorphic to a Drinfeld module defined over a finite field, we say that $\psi$ is \emph{isotrivial} (notice that this cannot happen in generic characteristic).

Every $a\in A$ induces an $\F_q$-linear map $\overline{F}\to \overline{F}$ via the linearized polynomial associated to $\psi_a$. The \emph{$a$-torsion} of $\psi$, denoted by $\psi[a]$, is the kernel of $\psi_a$. If $I\subseteq A$ is a non-zero ideal, the \emph{$I$-torsion} of $\psi$, denoted by $\psi[I]$, is $\bigcap_{a\in I}\psi[a]$. We will denote by $\psi[I^{\infty}]$ the union of all the $\psi[I^n]$, where $n$ runs over the naturals. It is one of the basic facts in the theory that $\psi[I]$ is a free $A/I$-module of rank $r$ whenever $\ker\gamma$ is relatively prime with $I$. We remark that a point $\alpha\in \overline{F}$ is in $\psi[I]$ for some non-zero ideal $I\subseteq A$ if and only if for every $b\in A$ the point $\alpha$ is preperiodic for the map $\psi_b$.

From now on, we let $F/\F_q(T)$ be a finite extension and
\begin{equation}\label{drinfeld_phi}
	\phi_T=\sum_{i=0}^r\phi_i\tau^i\in F\{\tau\} \mbox{ with }r \ge 1 \mbox{ and }\phi_0\phi_r\ne 0.
\end{equation}
With a slight abuse of notation, we will also denote by $\phi_T=\sum_{i=0}^r\phi_ix^{q^i}\in F[x]$ the associated linearized polynomial. The map $\phi_T$ can be thought of as the image of $T$ in a rank $r$ Drinfeld module $\phi\colon \F_q[T]\to F\{\tau\}$, where the structure map $\gamma\colon \F_q[T]\to F$ is defined by $\gamma(T)=\phi_0$ (hence why the use of $\phi_T$ as notation). Notice that the condition $\phi_0\ne 0$ ensures that if $\phi$ has special characteristic, the characteristic is not $(T)$.

\begin{remark}\label{new_drinfeld_module}
	Given a Drinfeld module $\phi\colon \F_q[T]\to F\{\tau\}$ of rank $r$, let $B$ be a maximal commutative subalgebra of $\en(\phi)^0$, and let $A$ be the integral closure of $\F_q[T]$ in $B$. Then by \cite[Proposition 4.1]{devic}, up to replacing $F$ by a finite extension there is a Drinfeld module $\widehat{\phi}\colon A\to F\{\tau\}$ of rank $r/de$ whose restriction to $\F_q[T]$ is $F$-isogenous to $\phi$ and whose endomorphism ring is $A$. If $\phi$ has generic characteristic, then it has complex multiplication if and only if $\widehat{\phi}$ has rank $1$, by definition, while if $\phi$ has special characteristic, then it is isotrivial if and only if $\widehat{\phi}$ has rank $1$ by \cite[Proposition 6.3]{devic}.
\end{remark}

In order to ease the exposition, from now on we will say that $\phi$ has \emph{potential rank $1$} if $de=r$, or equivalently if $\widehat{\phi}$ has rank $1$.
%

\begin{lemma}\label{drinfeld_cm}
	Let $\phi\colon \F_q[T] \to F\{\tau\}$ be a Drinfeld module as in \eqref{drinfeld_phi} and let $\alpha \in F$.
	\begin{enumerate}
		\item If $\phi$ has potential rank $1$ and $\alpha$ is preperiodic for $\phi_T$ then $G_\infty(\phi_T,\alpha)$ is virtually abelian.
		\item If there exist a finite extension $F'/F$, a non-constant rational function $f\in F'(x)$ and an increasing sequence of positive integers $\{N_k\}_{k\ge 1}$ such that for every $k\ge 1$ and every $z\in \phi_T^{-N_k}(\alpha)$ we have that $z$ is a pole of $f$ or $f(z)\in {F'}^{\ab}$, then $\phi$ has potential rank $1$ and $\alpha$ is preperiodic for $\phi_T$.
	\end{enumerate}
\end{lemma}

\begin{proof}
	From now on, we let $\widehat{\phi}\colon A \to F\{\tau\}$ be the Drinfeld module of Remark \ref{new_drinfeld_module}, that exists after replacing $F$ by a finite extension.
	
	$(1)$ If $\phi$ has special characteristic, then since it has potential rank $1$ it is isotrivial. This means that up to replacing $F$ by a finite extension, there is an $F$-isomorphism $\vartheta$ from $\phi$ to a Drinfeld module $\phi'$ defined over $\overline{\F}_q$. Of course isomorphic Drinfeld modules have the same torsion fields, and the field generated by the torsion of $\phi'$ is abelian since it is an algebraic extension of $\F_q$. Since $\alpha$ is torsion for $\phi$ we have that $F({\phi}_T^{-\infty}(\alpha))\subseteq F(\phi_{\text{tor}})$, and the claim follows.
	
	Suppose now that $\phi$ has generic characteristic. The Drinfeld module $\widehat{\phi}$ has rank $1$ and it is a well-known fact in the theory of Drinfeld modules that the extension $F(\widehat{\phi}_{\text{tor}})/F$ generated by all torsion points of $\widehat{\phi}$ is abelian. It follows that $F(({\widehat{\phi}|_{\F_q[T]}})_{\text{tor}})/F$ is abelian, and since isogenous Drinfeld modules have the same torsion fields, $F(\phi_{\text{tor}})/F$ is abelian as well. Since $\alpha$ is torsion for $\phi$ we have that $F({\phi}_T^{-\infty}(\alpha))\subseteq F(\phi_{\text{tor}})$, and the claim follows.

	$(2)$ Up to replacing $F$ by a finite extension, we can assume that $F'=F$. 
	
	First we show that $\widehat{\phi}$ has rank $1$. Let 
	$$\Gamma_T\coloneqq\gal(F(\widehat{\phi}[T^\infty])/F).$$
	The Galois action gives a homomorphism $\Gamma_T\to \prod_{\p\mid (T)}\GL_n(K_\p)$, where $n$ is the rank of $\widehat{\phi}$, the product runs over all prime ideals of $A$ that divide $(T)$ and $K_\p$ is the fraction field of ${A}_\p$. By \cite[Theorem 1.1]{pink4} and \cite[Theorem 0.1]{pink1}, if we choose a prime $\p$ dividing $(T)$ we have that the projection of the image of $\Gamma_T$ in $\GL_n(K_\p)$ is Zariski dense. Since there exists an $F$-isogeny $\mu\colon\phi\to\widehat{\phi}|_{\F_q[T]}$, we get that  $\Gamma_T'\coloneqq\gal(F(\mu(\phi[T^\infty]))/F)$ is Zariski dense in $\GL_n(K_\p)$. Now, the field extension $F(\phi_T^{-\infty}(\alpha))$ certainly contains $F(\phi[T^\infty])$, because if $z\in F^{\text{sep}}$ satisfies $\phi_T^n(z)=\alpha$ for some $n\ge 1$, $\phi_T^n(z+w)=\alpha$ for every $w\in \phi[T^n]$, and hence $F(\phi[T^n])\subseteq F(\phi_T^{-n}(\alpha))$. On the other hand by hypothesis we know that for every $k\ge 1$ the extension $F(f(z)\colon z\in \phi_T^{-N_k}(\alpha), z \mbox{ is not a pole of } f,k\ge 1)/F$ is abelian; this shows that for every $k\ge 1$ the Galois group of the extension $F(\phi_T^{-N_k}(\alpha))\cdot F^{\ab}/F^{\ab}$ has finite exponent, bounded by $(\deg f)!$. Since for every $n\ge 1$ there exists $k\ge 1$ such that $N_k\ge n$ and the Galois group of $F(\phi_T^{-N_k}(\alpha))\cdot F^{\ab}/F^{\ab}$ surjects onto that of $F(\phi_T^{-n}(\alpha))\cdot F^{\ab}/F^{\ab}$, it follows that the Galois group of $F(\phi_T^{-\infty}(\alpha))\cdot F^{\ab}/F^{\ab}$ has finite exponent, bounded by $(\deg f)!$. In turn, the same must hold true for $F(\phi[T^\infty])\cdot F^{\ab}/F^{\ab}$, and hence also for $F(\mu(\phi[T^\infty]))\cdot F^{\ab}/F^{\ab}$. This implies that every commutator in $\Gamma_T'$ has order bounded by $(\deg f)!$, finally implying that $n=1$, as the following very general claim shows.
	\begin{center}
		Let $n\ge 2$ be an integer, let $L$ be an infinite field and let $G$ be a Zariski dense subgroup of $\GL_n(L)$. Then for every $M\ge 1$, $G$ contains a commutator of order $>M$.
	\end{center}
	Let $N\coloneqq M!$ and consider the closed subvariety
	$$Z_N\coloneqq \{(g,h)\in \GL_n(L)\colon [g,h]^N=1\}.$$
	This is proper: if $a\in L$ is such that $a^N\ne 1$ then if $D=\text{diag}(a,1,\ldots,1)$ and $S$ swaps the first two basis vectors, then $[D,S]=\text{diag}(a,a^{-1},1,\ldots,1)$ has order $>M$. Since $G$ is Zariski dense in $\GL_n(L)$, then $G\times G$ is Zariski dense in $\GL_n(L)^2$ and hence it is not contained in $Z_N$. Thus there are $g,h\in G$ such that $[g,h]^{M!}\ne 1$.
	
	Next, we prove that $\alpha$ is preperiodic for $\phi_T$, or, equivalently, that $\alpha$ is torsion for $\phi$. Suppose, by contradiction, that $\alpha$ is not torsion. Put
	\[
	A_0\coloneqq\mathbb F_q[T]
	\qquad\text{and}\qquad
	M\coloneqq A_0\cdot\alpha\subseteq F,
	\]
	where the $A_0$-module structure is induced by $\phi$. Since $\alpha$
	is not torsion, $M$ is a free $A_0$-module of rank
	$1$.
	
	We first claim that the hypothesis of \cite[Theorem~6.7]{pink3} is satisfied for $(\phi,M)$. Choose data	$(A',F_1,\phi',\mu,B)$ as in \cite[Proposition~4.3(c)]{pink3}: thus $A'$ is the normalization
	of the center of $\operatorname{End}_{F^{\mathrm{sep}}}(\phi)$,
	$\phi'$ is a Drinfeld $A'$-module, and	$\mu\colon\phi\longrightarrow\phi'|_{A_0}$ is an isogeny. Put
	\[
	S\coloneqq
	\operatorname{End}_{F^{\mathrm{sep}}}\bigl(\phi'|_B\bigr).
	\]
	The aforementioned hypothesis is that the natural map
	\begin{equation}\label{eq:pink-injectivity}
		\begin{split}
			\iota\colon S\otimes_{A_0}M&\longrightarrow F^{\mathrm{sep}},\\
			\sum_i s_i\otimes m_i&\longmapsto
			\sum_i s_i\bigl(\mu(m_i)\bigr)
		\end{split}
	\end{equation}
	is injective. Since $M$ is freely generated by $\alpha$, the source of
	\eqref{eq:pink-injectivity} is naturally identified with $S$, and
	under this identification $\iota$ becomes
	\[
	S\longrightarrow F^{\mathrm{sep}},
	\qquad
	s\longmapsto s\bigl(\mu(\alpha)\bigr).
	\]
	If	$s\bigl(\mu(\alpha)\bigr)=0$ for some $0\neq s\in S$ then $\mu(\alpha)$ is contained in the finite subgroup $\ker(s)$. This is stable under the $B$-action, and hence it is torsion: there is some $0\ne b\in B$ such that $\phi'_b(\mu(\alpha))=0$.	Since $A'$ is finite over $A_0$, the nonzero ideal $bA'$ has nonzero
	contraction to $A_0$. Choose $0\neq a\in bA'\cap A_0$: then
	$$\phi'_a\bigl(\mu(\alpha)\bigr)=\mu\bigl(\phi_a(\alpha)\bigr)=0.$$
	This shows that $\phi_a(\alpha)\in \ker\mu$, and hence it is
	annihilated by some $0\neq c\in A_0$. It follows that $\alpha$ is annihilated by $ca$, contradicting our assumption.
		
	By \cite[Theorem~6.7]{pink3}, the adelic Kummer
	subgroup $\Delta_{\mathrm{ad},M}$ (that is the kernel of the restriction of the action on $T_{\mathrm{ad}}(\phi,M)$ to $T_{\mathrm{ad}}(\phi)$) is therefore open in
	\[
	\operatorname{Hom}_{A_0}
	\bigl(M,T_{\mathrm{ad}}(\phi)\bigr)
	\cong T_{\mathrm{ad}}(\phi).
	\]

	Let	$\mathfrak t\coloneqq(T)\subseteq A_0$ (notice that since $\phi_0=\gamma(T)\neq 0$, the prime $\mathfrak t$ is different	from the characteristic of $\phi$). Let	$T_{\mathfrak t}(\phi)
	=\varprojlim_n\phi[T^n]$ be the $\mathfrak t$-adic Tate module and let $V\subseteq T_{\mathfrak t}(\phi)$ be the projection of $\Delta_{\mathrm{ad},M}$ onto its	$\mathfrak t$-adic component, that is open. Consequently,
	there exists $m\geq 0$ such that
	\begin{equation}\label{eq:open-kummer-subgroup}
		T^mT_{\mathfrak t}(\phi)\subseteq V.
	\end{equation}
	
	Choose a compatible sequence of division points
	\[
	x_n\in\phi_{T^n}^{-1}(\alpha),
	\qquad
	\phi_T(x_{n+1})=x_n.
	\]
	The Galois action on the $\mathfrak t$-power division tower of
	$\alpha$ is described by an affine representation
	\[
	\begin{split}
	\widetilde{\rho}_{\alpha,\mathfrak t}\colon
	G_F&\longrightarrow
	T_{\mathfrak t}(\phi)\rtimes
	\operatorname{Aut}_{A_{0,\mathfrak t}}
	\bigl(T_{\mathfrak t}(\phi)\bigr)\\
	\sigma&\longmapsto (u_\sigma,\rho_{\mathfrak t}(\sigma)),
	\end{split}
	\]
	and the image of $u_\sigma$ modulo $T^n$ is characterized by
	\[
	\sigma(x_n)=x_n+u_{\sigma,n}.
	\]
	
	The image of the linear representation
	\[
	\rho_{\mathfrak t}\colon
	G_F\longrightarrow
	\operatorname{Aut}_{A_{0,\mathfrak t}}
	\bigl(T_{\mathfrak t}(\phi)\bigr)
	\]
	is nontrivial, since the prime-to-characteristic rational torsion is finite. We may therefore choose
	$\sigma\in G_F$ such that
	\[
	\gamma\coloneqq\rho_{\mathfrak t}(\sigma)\neq 1.
	\]
	Write
	\[
	\widetilde{\rho}_{\alpha,\mathfrak t}(\sigma)=(u,\gamma).
	\]
	
	For every $v\in V$, by definition of $\Delta_{\mathrm{ad},M}$ there exists $\tau_v\in G_F$ such that
	\[
	\widetilde{\rho}_{\alpha,\mathfrak t}(\tau_v)=(v,1).
	\]
	Since
	\begin{equation}\label{eq:affine-commutator}
		\bigl[(u,\gamma),(v,1)\bigr]
		=\bigl((\gamma-1)v,1\bigr),
	\end{equation}
	it follows that the $\mathfrak t$-adic translation parts of	commutators contain the subgroup
	\[
	W\coloneqq(\gamma-1)V
	\subseteq T_{\mathfrak t}(\phi).
	\]
	
	The group $W$ is infinite. Indeed, by
	\eqref{eq:open-kummer-subgroup} it contains $(\gamma-1)T^mT_{\mathfrak t}(\phi)$, and since $\gamma\neq 1$ surely there exists	$y\in T_{\mathfrak t}(\phi)$ such that	$w_0\coloneqq(\gamma-1)T^my\neq 0$.	As $\gamma$ is $A_{0,\mathfrak t}$-linear, we have that	$A_{0,\mathfrak t}w_0\subseteq W$.
	
	Now for every $n\geq 1$, let
	\[
	W_n\subseteq
	T_{\mathfrak t}(\phi)/T^nT_{\mathfrak t}(\phi)
	\cong\phi[T^n]
	\]
	be the image of $W$. Since $w_0\neq 0$, we obtain
	$$
		|W_n|\longrightarrow\infty
		\qquad\text{as }n\longrightarrow\infty.
	$$
	
	By \eqref{eq:affine-commutator}, for every $w\in W_n$ there is a
	commutator $c_w\in G_F$ such that
	\begin{equation}\label{eq:commutator-translation}
		c_w(x_n)=x_n+w.
	\end{equation}
	Let $d\coloneqq\deg(f)$; since $N_k\to\infty$ there exists $k$ such that $|W_{N_k}|>d$. Set $n=N_k$. By hypothesis, $f(x_n)\in\mathbb P^1(F^{\mathrm{ab}})$. On the other hand, every commutator in $G_F$ acts trivially on $F^{\mathrm{ab}}$, as	well as on the point $\infty$. Since $f$ is defined over $F$,
	\eqref{eq:commutator-translation} gives
	\[
	f(x_n+w)
	=f\bigl(c_w(x_n)\bigr)
	=c_w\bigl(f(x_n)\bigr)
	=f(x_n)
	\]
	for every $w\in W_n$. Hence
	\[
	\{x_n+w\colon w\in W_n\}\subseteq f^{-1}(f(x_n)),
	\]
	implying that $\#f^{-1}(f(x_n))	\geq |W_n|	>d$, and hence contradicting the fact that $\deg f=d$.
\end{proof}
Applying Lemma \ref{drinfeld_cm} with $f=x$ and $N_k=k$ for every $k\ge 1$ allows us to deduce the following corollary.

\begin{corollary}\label{abelian_gal_reps}
	$G_\infty(\phi_T,\alpha)$ is virtually abelian if and only if $\phi$ has potential rank $1$ and $\alpha$ is preperiodic for $\phi_T$. 
\end{corollary}

Although cases described in Lemma \ref{drinfeld_cm} furnish, in generic characteristic, a whole new set of examples of abelian dynamical Galois groups, these cannot be the whole story. In fact, similarly to what happens with Chebishev polynomials and Latt\`es maps, one can conceive a notion of "quotient map" for affine polynomials in positive characteristic, and these quotient maps can yield genuinely new examples.

\begin{definition}
	A rational map $\psi\in F(x)$ is called a \emph{Drinfeld-Latt\`es} map if it is semiconjugate to a map of the form $\phi_T+\delta$, where $\phi_T=\sum_{i=0}^r\phi_ix^{q^i}\in \overline{F}[x]$ with $r\ge 1$, $\phi_0\phi_r\ne 0$ and $\delta\in \overline{F}$ is a preperiodic point for $\phi_T$ or, equivalently, a torsion point for $\phi$. In other words, there exist $\phi_T,\delta$ and $f\in \overline{F}(x)$ of positive degree such that the following diagram commutes:
	
	\begin{equation} \label{semiconjugacy}
		\xymatrix{
			\P_{\overline{F}}^1 \ar[r]^{\phi_T+\delta} \ar[d]_{f} &\P_{\overline{F}}^1 \ar[d]^{f} \\
			\P_{\overline{F}}^1 \ar[r]^{\psi}       & \P_{\overline{F}}^1}.
	\end{equation}		
\end{definition}

\begin{remark}
	Every Drinfeld-Latt\`es map $\psi$ is $\overline{F}$-conjugate to a polynomial. In fact, suppose that $\psi\in F(x)$. If $f(\infty)=\infty$ then by noticing that $\deg\psi=\deg \phi_T$ and looking at ramification indices we see that $\infty$ must have ramification index $\deg \psi$ under $\psi$, which is equivalent to saying that $\psi$ is a polynomial. If $f(\infty)=a$ for some $a\in \overline{F}$, then up to replacing $f$ with $\frac{1}{x-a}\circ f$ and $\psi$ with $\psi'=\frac{1}{\psi\left(a+\frac{1}{x}\right)-a}$ we have that $f$ fixes $\infty$, and by the argument above we have that $\psi'$, that is the conjugate of $\psi$ via $\frac{1}{x-a}$, is a polynomial.
\end{remark}
\begin{example}
	Consider the Carlitz module given by $\phi_T=Tx+x^q$, and let $f=x^{q-1}$. Then
	$$f\circ\phi_T=(x(T+x^{q-1}))^{q-1},$$
	so that $\psi=x(T+x)^{q-1}$ is a Drinfeld-Latt\`es map.
\end{example}
In the number fields setting, one can show (see \cite[Theorem~13]{andrews}) that if $T_d$ is a Chebishev polynomial then $G_\infty(T_d,\alpha)$ is abelian if and only if $\alpha=\zeta+\zeta^{-1}$ for a root of unity $\zeta$, and the same holds true for $-T_d$. Moreover, if $\phi$ is a Latt\`es map then one can show (see \cite[Theorem~D]{ferostzan}) that $G_\infty(\phi,\alpha)$ is virtually abelian if and only if the elliptic curve associated with $\phi$ has complex multiplication and $\alpha$ is preperiodic for $\phi$ (that is equivalent to saying that it is the projection of a torsion point on the curve). The next theorem is an analogue of the aforementioned two results in the setting of Drinfeld-Latt\`es maps.

\begin{theorem}\label{abelian_quotients}
	Let $\psi\in F[x]$ be a Drinfeld-Latt\`es map fitting in a diagram as \eqref{semiconjugacy}, let $\alpha\in \overline{F}$ be a non-pole for $f$ and let $\beta=f(\alpha)$. Then the following are equivalent:
	\begin{enumerate}
		\item $G_\infty(\phi_T,\alpha)$ is virtually abelian.
		\item $G_\infty(\phi_T+\delta,\alpha)$ is virtually abelian.
		\item $G_\infty(\psi,\beta)$ is virtually abelian.
	\end{enumerate}
\end{theorem}
\begin{proof}
	Since every claim is about virtual abelianity, it is harmless to assume that $\phi$ and $f$ are both defined over $F$ and that $\alpha,\delta\in F$, and hence we will do so.
	
	$(1)\implies (2)$. Start by noticing that the sequence
	\begin{equation}\label{torsion_sequence}
		\left\{\sum_{i=0}^n\phi_T^i(\delta)\right\}_{n\ge 1}
	\end{equation}
	takes only finitely many values, and every element of it is a torsion point for $\phi$. In fact, since $\delta$ is torsion for $\phi_T$, in particular it belongs to $\phi[b]$ for some $b\in \F_q[T]$, and the latter is a finite group acted on by $\phi_T$. In particular, every element of the sequence \eqref{torsion_sequence} is a $b$-torsion point for $\phi$. Since
	\begin{equation}\label{iterates_of_translate}
		(\phi_T+\delta)^n=\phi_T^n+\sum_{i=0}^{n-1}\phi_T^i(\delta),
	\end{equation}
	and $\alpha$ is torsion for $\phi$ by Corollary \ref{abelian_gal_reps}, this proves that every element in $(\phi_T+\delta)^{-n}(\alpha)$ is a torsion point for $\phi$. It follows that $(\phi_T+\delta)^{-\infty}(\alpha)\subseteq F({\phi}_{\text{tor}})$. Since $G_\infty(\phi_T,\alpha)$ is virtually abelian, by Corollary \ref{abelian_gal_reps} we get that $\phi$ has potential rank $1$ and $\alpha$ is preperiodic for $\phi_T$. But again this shows that up to replacing $F$ by a finite extension, $F(\phi_{\text{tor}})\subseteq F^{\text{ab}}$, as mentioned in the proof of Lemma \ref{drinfeld_cm}.
	
	$(2)\implies (1)$. Since the sequence \eqref{torsion_sequence} takes only finitely many values, there is an $N\in \N$ and an increasing sequence of indices $\{n_k\}_{k\ge 1}$ such that if $\overline{\delta}\coloneqq \sum_{i=0}^{N}\phi_T^i(\delta)$ then $\sum_{i=0}^{n_k}\phi_T^i(\delta)=\overline{\delta}$ for every $k\ge 1$. By \eqref{iterates_of_translate} we get that
	\begin{equation}\label{fixed_subsequence}
		(\phi_T+\delta)^{n_k+1}-\alpha=\phi_T^{n_k+1}+\overline{\delta}-\alpha
	\end{equation}
	for every $k\ge 1$. Up to replacing $F$ by a finite extension, the hypothesis implies that every $z\in \phi_T^{-(n_k+1)}(-\overline{\delta}+\alpha)$ belongs to $F^{\text{ab}}$, and since this holds for every $k$ we deduce that $\phi_T^{-\infty}(-\overline{\delta}+\alpha)\subseteq F^{\text{ab}}$. By Corollary \ref{abelian_gal_reps} this shows that $\phi$ has potential rank $1$ and that $-\overline{\delta}+\alpha$ is preperiodic for $\phi_T$, which in turn implies that $\alpha$ is preperiodic, since $\overline{\delta}$ is. Applying Lemma \ref{drinfeld_cm} we conclude that $G_\infty(\phi_T,\alpha)$ is virtually abelian.	
	
	$(2)\implies (3)$. Since we have proven $(1)$ and $(2)$ to be equivalent, $G_\infty(\phi_T,\alpha)$ is virtually abelian, and hence by Corollary \ref{abelian_gal_reps} we have that $\phi$ has potential rank $1$ and $\alpha$ is preperiodic for $\phi_T$, and in turn it is preperiodic for $\phi_T+\delta$ as well. Notice that this implies that $\beta$ is preperiodic for $\psi$: since $(\phi_T+\delta)^k(\alpha)=(\phi_T+\delta)^\ell(\alpha)$ for some $k\ne \ell$ then by applying $f$ we get that $\psi^k(\beta)=\psi^\ell(\beta)$. Now let $n\in \N$ and $\beta_n\in \psi^{-n}(\beta)$. Let $\alpha_n\in \mathbb P^1(\overline{F})$ be such that $f(\alpha_n)=\beta_n$. We claim that $\alpha_n$ is preperiodic for $\phi_T+\delta$. If $\alpha_n=\infty$, this is obvious. Otherwise, if $\{(\phi_T+\delta)^m(\alpha_n)\}_{m\ge 1}$ was infinite, then so would be $\{f((\phi_T+\delta)^m(\alpha_n))\}_{m\ge 1}$, but the latter set coincides with $\{\psi^m(\beta_n)\}_{m\ge 1}$, and since $\psi^n(\beta_n)=\beta$ that is preperiodic for $\psi$, this cannot happen. Since $\alpha_n$ is preperiodic for $\phi_T+\delta$, then it is preperiodic for $\phi_T$; since $\phi$ has potential rank $1$ we get that up to replacing $F$ by a finite extension (independent of $n$) $F(\alpha_n)/F$ is abelian. Since $\beta_n\in F(\alpha_n)$ and $n$ was arbitrary, we get that $\psi^{-\infty}(\beta)\subseteq F^{\text{ab}}$, and the claim follows.
	
	$(3)\implies (1)$. By replacing $F$ by a finite extension, we can assume that $G_\infty(\psi,\beta)$ is abelian and that everything is defined over $F$. By fixing an increasing sequence of positive integers $\{n_k\}$ as in \eqref{fixed_subsequence} and considering the fact that $f$ induces a map $(\phi_T+\delta)^{-\infty}(\alpha)\to \psi^{-\infty}(\beta)$ (outside of its poles), with $\psi^{-\infty}(\beta)\subseteq F^{\ab}$ by hypothesis, we are in the position to apply Lemma \ref{drinfeld_cm} to the pair $(\phi,\alpha-\overline{\delta})$. We deduce that $\phi$ has potential rank $1$ and that $\alpha-\overline{\delta}$ is torsion for $\phi$. Since $\overline{\delta}$ is itself torsion, we deduce that $\alpha$ is preperiodic for $\phi$. A second application of the same lemma allows then to conclude.
\end{proof}

In view of Theorem \ref{main_thm}, Corollary \ref{abelian_gal_reps}, and Theorem \ref{abelian_quotients} we thus formulate a positive characteristic analogue of Conjecture \ref{ap_conjecture}. In order to do so we have to take care of a technical issue: in order to talk about the Galois group of $f^n-\alpha$ we need the latter to have roots in $F^{\text{sep}}$. For $f\in F[x]$ with $f'\not\equiv 0$ and $\alpha\in F$ we set
$$C_{f,\alpha}\coloneqq\{c\in \overline{F}\colon f'(c)=0 \mbox{ and } f^m(c)=\alpha\mbox{ for some }m\ge 1\}.$$
It is then elementary to see that $f^n-\alpha$ has all its roots in $F^{\text{sep}}$ for every $n\ge 1$ if and only if $C_{f,\alpha}\subseteq F^{\text{sep}}$. 
\begin{conjecture} \label{our conjecture}
	Let $F/\F_q(T)$ be a finite extension, let $f\in F[x]$ with $f'\not\equiv 0$ have degree $\ge 2$ and be such that $C_{f,\alpha}\subseteq F^{\text{sep}}$, and let $\alpha\in F$. Then $G_\infty(f,\alpha)$ is virtually abelian if and only if one of the following holds:
	\begin{enumerate}
		\item The pair $(f,\alpha)$ is $\overline{F}$-conjugate to a constant pair.
		\item The pair $(f,\alpha)$ is $\overline{F}$-conjugate to a pair $(\psi,\beta)$ where $\psi$ is a Drinfeld-Latt\`es map for a CM Drinfeld module $\phi$ of generic characteristic and $\beta$ is preperiodic for $\psi$.
	\end{enumerate} 
\end{conjecture}
We remark that if $\alpha$ is exceptional for $f$ then the pair $(f,\alpha)$ is necessarily $\overline{F}$-conjugate to a constant pair (see for example \cite[Theorem 1.19]{benedetto}), so we do not need to exclude exceptional points.
 \bibliographystyle{plain}
\bibliography{bibliography}
\end{document}